\documentclass{article}
\usepackage[utf8]{inputenc}
\usepackage{amssymb}
\usepackage{amsmath}
\usepackage{amsmath,amssymb}
\usepackage[utf8]{inputenc}
\usepackage{fourier, heuristica}
\usepackage{mathtools, nccmath}
\usepackage{mathtools}
\usepackage[english]{babel}
\usepackage{amsthm}
\usepackage{systeme}
\usepackage{geometry}[margin=1in]
\usepackage{authblk}
\newtheorem{theorem}{Theorem}[section]
\newtheorem{corollary}{Corollary}[theorem]
\newtheorem{lemma}[theorem]{Lemma}
\newtheorem{definition}[theorem]{Definition}

\newtheorem{prop}[theorem]{Proposition}

\newtheorem*{theorem*}{Theorem}
\newtheorem*{remark}{Remark}
\usepackage{graphicx}
\usepackage{relsize}
\usepackage{tcolorbox}
\usepackage{tikz-cd}
\usepackage[all]{xy}
\usepackage{url}
\numberwithin{equation}{section}
\usepackage{eqparbox}

\usepackage{hyperref, cleveref}

\usepackage{hyperref}
\hypersetup{
    colorlinks=true,
    linkcolor=blue,
    filecolor=magenta,      
    urlcolor=cyan,
    pdftitle={Overleaf Example},
    pdfpagemode=FullScreen,
    }

\makeatletter
\newcommand{\address}[1]{\gdef\@address{#1}}
\newcommand{\email}[1]{\gdef\@email{\url{#1}}}
\newcommand{\@endstuff}{\par\vspace{\baselineskip}\noindent\small
\begin{tabular}{@{}l}\scshape\@address\\\textit{E-mail address:} \@email\end{tabular}}
\AtEndDocument{\@endstuff}
\makeatother
\title{Counting Closed Geodesics in Rank 1 $\mathrm{SL}\left(2,\mathbb{R}\right)$-orbit Closures}
\author{John Rached}
\date{}
\address{Department of Mathematics, Purdue University,150 N. University Street, West Lafayette, IN 47907, United States}
\email{jabourac@purdue.edu}

\begin{document}
\maketitle

\begin{abstract}

We obtain bounds on the numbers of intersections between triangulations as the conformal structure of a surface varies along a Teichm{\"u}ller geodesic contained in an $\mathrm{SL}\left(2,\mathbb{R}\right)$-orbit closure of rank 1 in the moduli space of Abelian differentials. For $0 \leq \theta \leq 1$, we obtain an exponential bound on the number of closed geodesics in the orbit closure, of length at most $R$, that spend at least $\theta$-fraction of their length in a region with short saddle connections.
\end{abstract}

\section{Introduction}

Many problems in geometry and dynamical systems lead to the study of compact surfaces without boundary which admit flat metrics away from a finite set of conical singularities, where all the curvature is concentrated. A central and classical problem is the study of billiard trajectories, for which understanding flat surfaces obtained by gluing sides of a polygon is essential (see $\cite{wright2016rational}$). It may be very difficult to say anything about an individual flat surface, but there is a natural linear action on the space of deformations of a surface, giving rise to an $\mathrm{SL}\left(2,\mathbb{R}\right)$ dynamical system. The closure of the $\mathrm{SL}\left(2,\mathbb{R}\right)$-orbit of a specific flat surface can provide a wealth of information about the original surface. For a comprehensive survey on this perspective, see $\cite{zorich2006flat}$. \\

For a fixed surface with constant Gaussian curvature $K$, there are natural questions that can be asked about about the behavior of circles, discs, and geodesics. One can, in particular, ask about the growth of lengths of geodesics and growth of the cardinality, where geodesics are considered distinct if they stay an $\epsilon$-distance away from each other. In the case $K$ = -1 and using, for instance, the disk model of hyperbolic space \ $\mathbb{H}^2$, elementary calculations show that the area of an $r$-ball grows like $\pi e^r$, which allows one to conclude that the cardinality of $\epsilon$-separated geodesics emanating from a point has exponential growth. Any closed surface $M$ of constant curvature $K$ = -1 has $\mathbb{H}^2$ as its universal cover, and may be written as a quotient $M = \mathbb{H}^2/\Gamma$, where $\Gamma$ is the fundamental group of $M$. The fundamental group produces closed geodesics on $M$, so it is natural to ask about the behavior of closed geodesics. Furthermore, one can ask about the behavior of the growth of closed geodesics as the hyperbolic structure varies. The most interesting and difficult case of $\textit{simple}$ closed geodesics was the topic of Mirzakhani's thesis $\cite{mirzakhani2008growth}$.\\

The subject of this paper is to consider the growth of closed geodesics in the $\mathrm{SL}\left(2,\mathbb{R}\right)$-orbit closures of flat surfaces, for a natural metric on the space of deformations.

\subsection{Overview}

In a hyperbolic surface, the area growth of radius $R$ balls is roughly $e^R$, and the growth rate of $\epsilon$-separated length $R$ geodesics is $Ce^R$, where $C$ is a constant depending on $\epsilon$. This suggests that the number of closed geodesics of length $\leq R$ grows exponentially in $R$. Using the Selberg trace formula, Huber proved that for hyperbolic surfaces, the number of closed geodesics with length $\leq R$ is asymptotic to $\frac{e^R}{R}$. Later, using softer methods, Margulis achieved the same count. For $M$ a closed Riemannian manifold with negative sectional curvatures, Margulis showed in his thesis $\cite{margulis2003some}$ that if $\mathcal{N}(R)$ denotes the number of closed geodesics of lengths $\leq R$, and $h$ is the topological entropy of $M$

\begin{displaymath}
\mathcal{N}(R) \sim \frac{e^{hR}}{hR}
\end{displaymath}

as $R \to \infty$, where the notation $A \sim B$ means that the ratio $A/B$ tends to $1$ as $R$ tends to infinity.\\

The methods employed by Margulis crucially rely on the homogeneity of hyperbolic surfaces. In particular, they do not directly carry over to the setting of Teichm{\"u}ller dynamics. As Teichm{\"u}ller space exhibits, metrically, many of the features of negatively curved manifolds, it is natural to look for analogs of this geodesic count in this inhomogeneous context.\\

Let $S$ = $S_{g}$ be a surface of genus $g$, and let $\mathcal{M}_{g}(S)$ be the moduli space of Riemann surfaces homeomorphic to $S$. There is a natural identification of the cotangent bundle of $\mathcal{M}_{g}(S)$ with the space of finite area quadratic differentials on $S$. Let $\Omega\mathcal{M}_{g}(S)$ be the space of finite area quadratic differentials that are squares of Abelian differentials (holomorphic 1-forms $\omega$). Henceforth, we will use the term Abelian differential to refer to quadratic differentials that are squares of holomorphic 1-forms. Let $\Omega_1\mathcal{M}_g(S) \subset \Omega\mathcal{M}_g(S)$ be the unit-area locus of $\Omega\mathcal{M}_g(S)$. There is an $\mathrm{SL}(2,\mathbb{R})$-action on $\Omega_1\mathcal{M}_{g}(S)$ given by the action of $\mathrm{SL}(2,\mathbb{R})$ on $\mathbb{R}^2$ corresponding to "multiplication by $A \in \mathrm{SL}(2,\mathbb{R})$" on the atlas of charts to $\mathbb{R}^2$ determined by the quadratic differential. By Teichm{\"u}ller theory, the orbits of the diagonal flow $g_t = \begin{bmatrix}
e^t & 0\\
0 & e^{-t}
\end{bmatrix}$ project to Teichm{\"u}ller geodesics under the natural projection map $\pi:\Omega_1\mathcal{M}_{g}(S) \rightarrow \mathcal{M}_{g}(S)$.  $\Omega\mathcal{M}_{g}(S)$ has a natural stratification: we say $(X,\omega^2) \in \Omega\mathcal{M}_{g,n}(S)$ is of type $ \left(p_1,...,p_k\right)$ if $\omega$ has zeroes of order $\left\{p_i\right\}$. The space of all Abelian differentials of type $\left(p_1,...,p_k\right)$ in $\Omega\mathcal{M}_{g}(S)$ will be denoted by $\Omega\mathcal{M}_{g}(S)\left(p_1,...,p_k\right)$. This is a complex-analytic orbifold of real dimension $4g + 2k - 1$. It also has the structure of a complex algebraic variety.\\

 Using the ideas in Margulis' thesis, Eskin and Mirzakhani $\cite{eskin2008counting}$ show that for $\mathcal{N}(R)$ the number of closed Teichm{\"u}ller geodesics in $\Omega_1\mathcal{M}_{g}(S)$, $\mathcal{N}(R) \sim \frac{e^{hR}}{hR}$ as $R \to \infty$ where $h = 6g -6$. In$\cite{eskin2019counting}$, Eskin-Mirzakhani-Rafi extend this to connected components $\mathcal{C} \subset \Omega_1\mathcal{M}_{g}(S)\left(p_1,...,p_k\right)$ of strata of type $\left(p_1,...,p_k\right)$ (in fact, to general quadratic differentials). The $h$ in the exponent in this case is related to the real dimension of the connected component. Specifically, if $h = \frac{1}{2}\left[1 + \mathrm{dim}_{\mathbb{R}}\mathcal{C}\right]$, the main theorem in $\cite{eskin2019counting}$ is

\begin{theorem*}[Eskin-Mirzakhani-Rafi]
$\mathcal{N}(\mathcal{C},R) \sim \frac{e^{hR}}{hR}$
as $R \to \infty$ and $\mathcal{N}(\mathcal{C},R)$ is the number of closed geodesics of lengths $\leq R$ in a connected component $\mathcal{C}$ in a stratum.
\end{theorem*}

Let $\mathcal{T}_g(S)$ be the Teichm{\"u}ller space, the orbifold universal cover of $\mathcal{M}_g(S)$. Let $\Omega\mathcal{T}_g(S)$ and $\Omega_1\mathcal{T}_g(S)$ be defined similarly as above. We have an orbifold covering map:

\begin{displaymath}
\pi:\Omega_1\mathcal{T}_g(S)\left(p_1,...,p_k\right) \rightarrow \Omega_1\mathcal{M}_g(S)\left(p_1,...,p_k\right).
\end{displaymath}
Moreover, if $V \subset \Omega_1\mathcal{M}_g(S)\left(p_1,...,p_k\right)$ is a properly immersed manifold, the pre-image $\pi^{-1}(V)$ is a properly immersed manifold of the same dimension. We will be interested in special classes of submanifolds $V$ of $\Omega_1\mathcal{M}(S)\left(p_1,...,p_k\right)$, the first of which are the lowest dimensional $\mathrm{SL}(2,\mathbb{R})$-invariant manifolds that can be cut out locally by real-linear equations. Specifically, for $(X,\omega^2) \in \Omega_1\mathcal{T}_g(S)\left(p_1,...,p_k\right)$, we are interested in the orbit closure of $\pi\left(\mathrm{SL}(2,\mathbb{R})\cdot \omega^2\right).$ For $\mathrm{Mod}(S)$ the mapping class group of $S = S_g$, it follows from a classical result of of Smillie-Weiss $\cite{smillie2004minimal}$ that if $\Gamma_{\omega} \leq \mathrm{Mod}(S)$ is the set-wise stabilizer of the lift of $\pi\left(\mathrm{SL}(2,\mathbb{R})\cdot \omega^2\right)$ to the connected component $\widetilde{{\pi^{-1} \circ {\pi\left(\mathrm{SL}(2,\mathbb{R})\cdot \omega^2\right)}}}$ of 
$\Omega_1\mathcal{T}_g(S)\left(p_1,...,p_k\right)$ containing $(X,\omega^2)$, then

\begin{displaymath}
\widetilde{{\pi^{-1}\circ{\pi\left(\mathrm{SL}(2,\mathbb{R})\cdot \omega^2\right)}}} =  \Gamma_w \cdot \left(\overline{\mathrm{SL}(2,\mathbb{R})\cdot \omega^2}\right)
\end{displaymath}
when $(X,\omega^2)$ is a $\textit{Veech}$ surface. Veech surfaces have special dynamical properties. In particular, for such $\omega$, $\pi\left(\mathrm{SL}(2,\mathbb{R})\cdot \omega^2\right)$ is a closed orbit. Note that $\mathrm{SL}(2,\mathbb{R})$ acts on $\Omega_1\mathcal{T}_g(S)\left(p_1,...,p_k\right)$ because the action is area-preserving, but there are in fact  $\mathrm{GL}(2,\mathbb{R})^+$-actions on $\Omega\mathcal{T}_g(S)\left(p_1,...,p_k\right)$ and $\Omega\mathcal{M}_g(S)\left(p_1,...,p_k\right)$. Furthermore, the image of $\mathrm{SL}(2,\mathbb{R})\cdot \omega^2$ under

\begin{displaymath}
\Omega_1\mathcal{T}_g(S)\left(p_1,...,p_k\right) \xrightarrow{\pi} \Omega_1\mathcal{M}_g(S)\left(p_1,...,p_k\right) \rightarrow \mathcal{M}_g(S)
\end{displaymath}
and the image of $\mathrm{GL}(2,\mathbb{R})^+\cdot \omega^2$ under

\begin{displaymath}
\Omega\mathcal{T}_g(S)\left(p_1,...,p_k\right) \xrightarrow{\pi} \Omega\mathcal{M}_g(S)\left(p_1,...,p_k\right) \rightarrow \mathcal{M}_g(S)
\end{displaymath}
are the same algebraic curve $V$ in $\mathcal{M}_g(S)$ and a corollary of our result can be interpreted as counting closed geodesics in $V$. The $\mathrm{SL}(2,\mathbb{R})$-orbit of $\omega^2$ in $\Omega_1\mathcal{T}_g(S)\left(p_1,...,p_k\right)$ is contained in the $\mathrm{GL}(2,\mathbb{R})^+$-orbit of $\omega^2$ in $\Omega\mathcal{T}_g(S)\left(p_1,...,p_k\right)$. In fact, we show that our observations apply to a more general class of higher $h$-dimensional affine invariant manifolds $V$ for which $\pi\left(\mathrm{SL}\left(2,\mathbb{R}\right)\cdot \omega^2\right)$ is $\textit{not}$ closed. If $(X,\omega^2)$ is such that $\overline{\mathrm{SL}\left(2,\mathbb{R}\right) \cdot \omega^2}$ is a rank-1 $\mathrm{SL}\left(2,\mathbb{R}\right)$-orbit closure, and $\Gamma_{\omega}$ is the stabilizer of the connected component of $\pi^{-1}\circ \overline{\pi\left(\mathrm{SL}\left(2,\mathbb{R}\right)\cdot \omega^2\right)}$ containing $(X,\omega^2)$, then $\pi^{-1}\circ \overline{\left(\mathrm{SL}\left(2,\mathbb{R}\right)\cdot \omega^2\right)}$ is the closure of $\Gamma_{\omega}\cdot \left(\mathrm{SL}\left(2,\mathbb{R}\right)\cdot \omega^2\right)$ in $\Omega\mathcal{T}_g(S)\left(p_1,...,p_k\right)$.\\

We review the definition of $\textit{extremal length}$ for a saddle connection on $\omega^2$, given in $\cite{eskin2019counting}$. For $\epsilon > 0$, and for any Abelian differential $(X,\omega^2) \in \tilde{V} = \overline{\mathrm{GL}(2,\mathbb{R})^{+}\cdot \omega^2}$, let $Q_{\omega}(\epsilon)$ be the set of saddle connections $\gamma$ so that either the extremal length with respect to $\omega^2$, $\mathrm{Ext}_{\omega}(\gamma)$, satisfies $\mathrm{Ext}_{\omega}(\gamma) \leq \epsilon$ or $\gamma$ appears in a geodesic representative of a simple closed curve $\alpha$ with $\mathrm{Ext}_{\omega}(\alpha)\leq \epsilon$. Let $\mathcal{Q}_{j,\epsilon}\left(p_1,...,p_k\right) \subset \Omega_1\mathcal{T}_g(S)\left(p_1,...,p_k\right)$ be the set of Abelian differentials containing at least $j$ homologically independent saddle connections. Let $V_{j,\epsilon}$ be the set of points in the $h$-dimensional affine manifold $V$ whose lift to $\Omega_1\mathcal{T}_g(S)\left(p_1,...,p_k\right)$ lies in $\mathcal{Q}_{j,\epsilon}\left(p_1,...,p_k\right)$. The main result of this paper is:

\begin{theorem}
\label{Theorem1.1}
Let $N_{\theta}(V_{j,\epsilon},R)$ be the number of closed geodesics of length at most R in $V$ that spend at least $\theta$-fraction of their length in $V_{j,\epsilon}$ where $0\leq\theta\leq1$. Given $\delta > 0$, there exist $\epsilon > 0$ and $R >0$ so that, for all $j\geq 0$ and $0 \leq \theta \leq 1$

\begin{displaymath}
N_{\theta}(V_{j,\epsilon},R) \leq e^{(h-j\theta + \delta)R}.
\end{displaymath}

\end{theorem}

This result demonstrates that the methods of $\cite{eskin2019counting}$, and the overall strategy employed by Margulis, can be used to obtain counts of closed geodesics inside $\mathrm{SL}\left(2,\mathbb{R}\right)$-orbit closures in strata. The seminal works of Eskin-Mirzakhani $\cite{eskin2018invariant}$ and Eskin-Mirzakhani-Mohammadi $\cite{eskin2015isolation}$ show that these orbit closures are affine invariant manifolds, locally cut out by $\mathbb{R}$-linear equations in period coordinates. We demonstrate, that in the simplest case of closed orbits, and the more difficult case of higher dimensional rank-1 orbit closures, the intersection theory of saddle connections is robust enough to leverage the constraints imposed by the linear equations defining the orbit closure to generalize the results of Eskin-Mirzakhani-Mohammadi. Our proof very closely follows that of Eskin-Mirzakhani-Rafi $\cite{eskin2019counting}$, and most of the methods they introduced carry over to our setting. It has been conjectured by Wright $\cite{wright2020tour}$ that a version of the asymptotic growth of the number of closed geodesics in strata in $\cite{eskin2019counting}$ should hold for general $\mathrm{SL}(2,\mathbb{R})$-closures in strata. In forthcoming work, we plan to use Theorem $\ref{Theorem1.1}$ to prove this conjecture for the lowest-dimensional $\mathrm{SL}(2,\mathbb{R})$-orbit closures (obtaining a version of Huber and Margulis' result for non-compact hyperbolic surfaces) and a significant generalization to the case of rank-1 $\mathrm{SL}(2,\mathbb{R})$ orbit closures.\\

This paper is organized as follows. In Section $\ref{Section2}$, we review the essential notions from dynamics on Veech surfaces and the geometry of Teichm{\"u}ller space. In Section $\ref{Section3}$, we study the intersection patterns for the $(q,\tau)$-regular triangulations introduced in $\cite{eskin2019counting}$ in the context of completely periodic surfaces. Finally, it is demonstrated in Section $\ref{Section4}$ that the results of Section $\ref{Section3}$ can be used to make Eskin-Mirzakhani-Raf's random walk argument work in this setting. In light of the relations between intersection numbers obtained in Section $\ref{Section3}$, there is very little new to check to extend the techniques from $\cite{eskin2019counting}$ to the setting of Teichm{\"u}ller curves. When the arguments carry over verbatim, we will only provide sketches of proofs, for the sake of clarity of exposition.

\section{Preliminaries}
\label{Section2}

\subsection{Teichm{\"u}ller metric and Teichm{\"u}ller curves}

In this subsection, we briefly review the metric structures we will be interested in, and the dynamics of Veech surfaces. All of the material here is standard, and we will sometimes follow $\cite{mcmullen2003billiards}$.
Let $S_{g}$ be a surface of genus $g$, satisfying $3g-3 > 0$. Let $\phi_1:S_g \rightarrow S_1$, $\phi_2:S_g \rightarrow S_2$, be two diffeomorphisms, where $S_1$ and $S_2$ have finite area. Denote by $\sim$ the equivalence relation
where $\phi_1 \sim \phi_2$ if $\phi_1 \circ \phi^{-1}_2$ is isotopic to a holomorphic diffeomorphism. Define

\begin{displaymath}
\mathcal{T}_{g}(S) = \big\{\phi:S_{g} \rightarrow S \ | \ \text{S is finite area }, \phi \ \text{a diffeomorphism}\big\} \big/ \sim.
\end{displaymath}
Given $\psi \in \text{Diff}^{+}\left(S_{g}\right)$, the group of orientation-preserving diffeomorphisms of $S_{g}$, we obtain a map $\mathcal{T}_{g}(S) \rightarrow \mathcal{T}_{g}(S)$ by precomposition by $\psi$. By the definition of the equivalence
relation $\sim$, if $\psi$ is contained in the normal subgroup $\text{Diff}^{0}\left(S_{g}\right)$ of diffeomorphisms isotopic to a holomorphic diffeomorphism, this map defined by precomposition is trivial. Therefore, we have an action of the mapping class group of $S_{g}$ 

\begin{displaymath}
\text{Mod}\left(S_{g} \right):= \text{Diff}^{+}\left(S_{g}\right) \big/ \text{Diff}^{0}\left(S_{g}\right)
\end{displaymath}
on $\mathcal{T}_{g}(S)$. The moduli space $\mathcal{M}_{g}(S)$ of $S_{g}$ is defined as 

\begin{displaymath}
\mathcal{M}_{g}(S):= \mathcal{T}_{g}(S)\big/ \text{Mod}\left(S_{g}\right).
\end{displaymath}

$\mathcal{T}_{g}(S)$ is a real-analytic space, and is homeomorphic to $\mathbb{R}^{6g-6}$. The cotangent space at $X \in \mathcal{T}_{g}(S)$ can be identified with $Q(X)$, the space of holomorphic quadratic differentials on $X$. The dual space (tangent space at $X$) is naturally identified with the space of harmonic Beltrami differentials on $X$, ${B}(X)$. Suppose $\mu = \mu(z)d\bar{z}\otimes(dz)^{-1} \in B(X)$ and $q = q(z)dz^2 \in Q(X)$. There is a natural non-degenerate pairing on $\mathcal{T}_g(S)$
\begin{displaymath}
\langle{\mu,q\rangle} = \int_X \mu(z)q(z)|dz|^2.
\end{displaymath}
The norm $||\mu||_T = \mathrm{sup}\left\{|\langle\{\mu,q\rangle\}:q \in Q(X), ||q||_1 = \int_X|q| = 1\right\}$ gives the $\textit{Teichm{\"u}ller metric}$ \ on $\mathcal{T}_g(S)$. This is the infinitesimal description of the Teichm{\"u}ller metric. The distance between two points in Teichm{\"u}ller space has a simple description in terms of dilatation. We have

\begin{displaymath}
d_{\mathcal{T}}\left(\left(\phi_1:S_g \rightarrow S_1\right),\left(\phi_2:S_g \rightarrow S_2\right)\right) = \frac{1}{2}\underset{\phi}{\mathrm{inf}} \; \mathrm{log} \; K_{\phi}
\end{displaymath}
where $\phi: X_1 \rightarrow X_2$ ranges over all quasi-conformal maps from $X_1$ to $X_2$ isotopic to $\phi_1\circ\phi_2^{-1}$ and $K_{\phi} \geq 1$ is the dilatation coefficient. Since the definition of the metric only depends on the holomorphic structure at $X$, the Teichm{\"u}ller metric on $\mathcal{T}_g(S)$ descends to a metric on $\mathcal{M}_{g}(S)$.  \\

It is well-known that a pair $(X,q)$ with $ q \in Q(X)$ and $q \neq 0$, generates a holomorphic embedding from the hyperbolic plane $\mathbb{H}$ with Kobayashi metric on $\mathbb{H}$ and Teichm{\"u}ller metric on $\mathcal{T}_g(S)$

\begin{displaymath}
\tilde{f}:\mathbb{H} \rightarrow \mathcal{T}_g(S).
\end{displaymath}
$\tilde{f}$ is an isometry. Passing to $\mathcal{M}_g = \mathcal{T}_g(S) \big/ \mathrm{Mod}(S_g)$, we obtain a $\textit{complex geodesic}$

\begin{displaymath}
f: \mathbb{H} \rightarrow \mathcal{M}_g(S).
\end{displaymath}
We summarize the constructions of $\tilde{f}$ and $f$, and the connection to dynamics. For $t$ a complex parameter, the Riemann surface $X_t = \tilde{f}(t)$ is characterized by the property that the complex dilatation $\mu_t$ of the extremal quasi-conformal map $\psi_t:X \rightarrow X_t$ is given by

\begin{equation}
\label{equation2.1}
\mu_t = \left(\frac{i-t}{i+t}\right) \cdot \frac{\bar{q}}{|q|}.
\end{equation}
Now, assume $q$ is the square of a holomorphic $1$-form $w$, $q = \omega^2$. Then, these complex geodesics have an interpretation in terms of the $\mathrm{SL}(2,\mathbb{R})$ action on $\Omega_1\mathcal{T}_g\left(p_1,...,p_k\right)$. In particular, any holomorphic $1$ -form $\omega$ yields, away from its zeroes, a flat metric Euclidean metric $|\omega|$ and an atlas of charts $U_i \rightarrow \mathbb{C}$ whose transition functions are translations. The $\mathrm{SL}(2,\mathbb{R})$-action on $\Omega_1\mathcal{T}_g\left(p_1,...,p_k\right)$ is defined by composition of these charts with matrices $A \in \mathrm{SL}(2,\mathbb{R})$ acting linearly on $\mathbb{C} \simeq \mathbb{R}^2$. $A \in \mathrm{SL}(2,\mathbb{R})$ can be interpreted as an affine map in these coordinate charts $A\cdot(X,\omega) = (X',\omega')$. This yields a quasi-conformal map $f: X \rightarrow X'$ with terminal holomorphic $1$-form $\omega'$. In the case that $A \in \mathrm{SO}(2,\mathbb{R})$, the map $f$ is conformal, so the action descends to a faithful action of $\mathrm{SL}(2,\mathbb{R})/\mathrm{SO}(2,\mathbb{R})$. One can show that this descended action induces an isometric injection

\begin{displaymath}
\mathrm{SL}(2,\mathbb{R})/\mathrm{SO}(2,\mathbb{R})(X,\omega) \rightarrow \mathcal{T}_g(S).
\end{displaymath}
Furthermore, composition with the projection $\pi:\mathcal{T}_g(S) \rightarrow \mathcal{M}_g(S)$ defines the complex geodesic

\begin{displaymath}
f: \mathbb{H} \rightarrow \mathcal{M}_g(S).
\end{displaymath}
One can check that if $t \in \mathbb{C}$ parametrizes the Riemann surfaces, the complex dilatation of the extremal quasi-conformal map $\psi_t: X \rightarrow X_t$ is given by

\begin{equation}
\mu_t = \left(\frac{i-t}{i+t}\right) \cdot \frac{\bar{\omega}}{|\omega|}
\end{equation}
which accords with Equation $\ref{equation2.1}$. Furthermore, $f$ factors through the quotient space $V = \mathbb{H} / \Gamma$, where $\Gamma = \left\{B \in \mathrm{Aut}(\mathbb{H}): f(B\cdot t) = f(t) \ \forall t \right\}$. In the case that $\Gamma$ is a lattice, we call the quotient map

\begin{displaymath}
f:V \rightarrow \mathcal{M}_g
\end{displaymath}
a Teichm{\"u}ller curve. $f$ is proper and generically injective. The image $f(V)$ is not a normal subvariety, but we will not distinguish between the Teichm{\"u}ller curve as a map $f$ and the normalization of $f(V)$.\\

Each non-zero holomorphic $1$-form $\omega$ determines a foliation $\mathcal{F}(\omega)$ of $X$ with singular points at the zeroes of $\omega$. There are two types of leaves $L$ of this foliation:

\begin{enumerate}
    \item Compact leaves, which are either saddle connections or closed loops.
    \item Noncompact leaves (infinite trajectories), which are copies of $\mathbb{R}$ or $[0,\infty)$ immersed in $X$.
\end{enumerate}

We say a non-compact leaf $L$ is $\textit{uniformly distributed}$ on $X$ if for any constant-speed (in the flat metric induced by $\omega$) immersion $\gamma:[0,\infty) \rightarrow X$, and any open set $U \subset X$ we have

\begin{displaymath}
\lim_{T \to \infty} \frac{|t \in [0,T]: \gamma(t) \in U|}{T} = \frac{\int_U \omega}{\int_X \omega}.
\end{displaymath}
We call a direction $\theta$ in which all infinite trajectories are uniformly distributed a $\textit{uniquely ergodic}$ direction. This leads us to a classical theorem of Veech, which will be the basis for our computations on bounds of intersection numbers.

\begin{theorem}[Veech,\cite{veech1989teichmuller}]
\label{theorem2.1}
Let $(X,\omega^2)$ generate a Teichm{\"u}ller curve. Then, for any  \ $\theta$, either

\begin{enumerate}
    \item All leaves of $\mathcal{F}(e^{i\theta}\omega)$ are compact, or
    \item $\theta$ is a uniquely ergodic direction.
\end{enumerate}
\end{theorem}
Moreover, we have the following corollary.

\begin{corollary}
\label{corollary2.1}
If $(X,\omega^2)$ generates a Teichm{\"u}ller curve, and has a saddle connection in the direction \ $\theta$, then $X$ admits a cylinder decomposition in the direction of \ $\theta$ where each cylinder is maximal (bounded by saddle connections). 
\end{corollary}
If a surface $(X,\omega^2)$, not necessarily generating a Teichm{\"u}ller curve, satisfies the conclusions of Theorem $\ref{theorem2.1}$, we say the surface is $\textit{completely periodic}$.

\begin{subsection}{Period Coordinates and Holonomy Vectors}

We briefly recall the relationship between the equations satisfied in local period coordinates as the Teichmuller curve is traversed, and the relations in homology that we will use.
For any $(X,\omega^2) \in \Omega\mathcal{T}_g(S)\left(p_1,...,p_k\right)$, there exists a triangulation $T$ of the underlying surface $X$ by saddle connections. Denote the zeroes of $\omega^2$ by $Z(\omega)$, where $|Z(\omega)|=k$. One can choose $h = 2g + |Z(\omega)| - 1$ directed edges $\left\{w_i\right\}_{i=1}^h$ of $T$, and an open neighborhood $U \subset \Omega\mathcal{T}_g(S)(p_1,...,p_l)$ of $(X,\omega^2)$ in its stratum, so that there exists an open analytic embedding:

\begin{displaymath}
\phi_{T,\omega}: U \rightarrow H^1(X,Z(\omega)) \cong \mathbb{C}^{2g + Z(\omega) -1}
\end{displaymath}
called the period map. It is defined by $(X,\omega^2)$ to be the relative class $[\omega] \in H^1(X,Z(\omega))$ which satisfies
$\left<[\omega],w_i\right> = \int_{w_i}\omega$ for all $w_i$ viewed as relative classes in $H_1(X,Z(\omega))$. For any other geodesic triangulation $T'$, the map $\phi_{T',\omega}\circ \phi^{-1}_{T,\omega}$ is linear.\\

We say a closed subset $M \subset \Omega\mathcal{T}_g(p_1,...,p_k)(S)$ is locally defined by real linear equations if, for each point $(X,\omega^2) \in M$ and open neighborhood $U$ as above, there exists a finite set of complex vector spaces $S_i \subset H^1(X,Z(\omega))$, invariant under complex conjugation, so that:

\begin{displaymath}
\phi_{T,\omega}(U\cap M) = \phi_{T,\omega}(U)\cap \left(\cup_{i=1}^k S_i \right).
\end{displaymath}
A consequence of the uniformization theorem is that every element $(X,\omega^2) \in \Omega\mathcal{T}_g(S)\left(p_1,...,p_k\right)$ can be presented in the form

\begin{displaymath}
(X,\omega^2) = (P,dz)\big/ \sim
\end{displaymath}
for a polygon $P \subset \mathbb{C}$, and the $1$-form $dz$ on $\mathbb{C}$. The reason for this is that one can construct a geodesic triangulation of the flat surface $(X,|\omega|)$, with $Z(\omega)$ among its vertices. $X$ can then be presented as a quotient $\sim$ of the edges of a collection of triangles. One can check that the periods $\int_{w_i} \omega$ defining $\phi_{T,\omega}$ correspond to the vectors $\left\{v_i\right\} \in \mathbb{C}$, the edges of the polygon $P$. Furthermore, the linear relations with real coefficients that these vectors satisfy correspond to the linear relations cutting out the complex vector subspaces $S_i \subset H^1(X,Z(\omega))$.

\end{subsection}

\begin{subsection}{General orbit closures and complete periodicity}

Henceforth, we will assume, without loss of generality, that when a closed subset $V \subset \Omega \mathcal{T}_g(p_1,...,p_k)(S)$ is locally defined by real linear equations, the finite subset $\left\{S_i\right\}$ of complex subspaces of $H^1(X,Z(\omega))$ consists of a single vector subspace $S$. Since $S$ comes with a complex conjugation, $S$ splits as

\begin{displaymath}
S = S^{\mathbb{R}} \oplus iS^{\mathbb{R}}
\end{displaymath}
where $S^{\mathbb{R}}:= S \ \cap H^1(X,Z(\omega),\mathbb{R})$. If $p:H^1(X,Z(\omega),\mathbb{R}) \rightarrow H^1(X,\mathbb{R})$, the restriction of the intersection form of $H^1(X,\mathbb{R})$ to $p\left(S^{\mathbb{R}}\right)$ is non-degenerate. Thus, there is a number $0 < r \leq 2g$ such that $\mathrm{dim}_{\mathbb{R}}p\left(S^{\mathbb{R}}\right) = 2r$. We call $r$ the $\textit{rank}$ of $V$.

\begin{theorem}[Eskin-Mirzakhani-Mohammadi,\cite{eskin2015isolation}]
For any $(X,\omega^2) \in \Omega\mathcal{T}_g(S)(p_1,...,p_k) $, the orbit closure \ $\overline{\mathrm{SL}\left(2,\mathbb{R}\right)\cdot \omega^2}$ is locally defined by real linear equations.
\end{theorem}

\begin{theorem}[Wright,\cite{wright2015cylinder}]
Any \ $\mathrm{SL}\left(2,\mathbb{R}\right)$-orbit closure $V = \overline{\mathrm{SL}\left(2,\mathbb{R}\right)\cdot \omega^2}$ of rank \ $1$ is completely periodic.
\end{theorem}

\end{subsection}

\begin{subsection}{Extremal lengths and thick subsurfaces}

We collect a series of useful results referenced in $\cite{eskin2019counting}$ to establish the bounds necessary between intervals of the random walk.
By a curve, we will mean the free homotopy class of a non-trivial, simple closed curve on $S_g$. Given a curve $\alpha$ on $S_g$ and $X \in \mathcal{T}_g(S)$, let $\ell_X(\alpha)$ denote the hyperbolic length of the unique geodesic in the homotopy class of $\alpha$. The extremal length of a curve $\alpha$ on $X$ is defined by:

\begin{displaymath}
\mathrm{Ext}_X(\alpha):= \underset{\rho}{\mathrm{sup}}\frac{\ell_{\rho}(\alpha)^2}{\mathrm{Area}(X,\rho)},
\end{displaymath}
where $\underset{\rho}{\mathrm{sup}}$ is taken over all metrics conformally equivalent to $X$. Denoting the set of curves on $S_g$ by $\textbf{S}_g$, we have the following relationship between extremal length and Teichm{\"u}ller distance:

\begin{theorem}[Kerchoff,\cite{kerckhoff1978asymptotic}]
For $X,Y \in \mathcal{T}_g(S)$,

\begin{displaymath}
d_{\mathcal{T}}(X,Y) = \underset{\beta \in \textbf{S}_g}{\mathrm{sup}} \mathrm{log} \left(\frac{\sqrt{\mathrm{Ext}_X(\beta)}}{\sqrt{\mathrm{Ext}_Y(\beta)}}\right).
\end{displaymath}
\end{theorem}
It will be useful for our purposes to have a simple description of the Teichm{\"u}ller distance, in the region of Teichm{\"u}ller space where simple closed curves have small extremal length. Let $\mathcal{A} = \left\{\alpha_1,...,\alpha_j\right\}$ be a collection of disjoint simple closed curves, with $\mathcal{A} \subset \textbf{S}_g$. Fix $\epsilon > 0$, and define:

\begin{displaymath}
\mathcal{T}^{\epsilon}(\mathcal{A}) = \left\{X \in \mathcal{T}_g(S) \; | \; \mathrm{Ext}_X(\alpha_i) \leq \epsilon \; \forall j \right\}.
\end{displaymath}
Using Fenchel-Nielsen coordinates on $\mathcal{T}_g(S)$, we have a map $\varphi_{\mathcal{A}}: \mathcal{T}^{\epsilon}(\mathcal{A}) \rightarrow (\mathbb{H})^j$ defined by

\begin{displaymath}
\varphi_{\mathcal{A}}(X) = \left\{\left(\theta_i(X),\frac{1}{\ell_X(\alpha_i)}\right)\right\}_{i=1}^j
\end{displaymath}
where $\theta_i$ is the Fenchel-Nielsen twist coordinate around $\alpha_i$. This furnishes a map

\begin{displaymath}
\varphi:\mathcal{T}^{\epsilon}(\mathcal{A}) \rightarrow (\mathbb{H})^j \times \mathcal{T}_g(S \setminus \mathcal{A})
\end{displaymath}
where $\mathcal{T}_g(S \setminus \mathcal{A})$ is the Teichm{\"u}ller quotient space obtained by collapsing all $\left\{\alpha_i\right\}_{i=1}^j$. We have the following result of Minsky.

\begin{theorem}[Minsky,\cite{minsky1996extremal}]
\label{theorem2.3}
There exists $\varepsilon_0 > 0$ such that for all $X, Y \in \mathcal{T}^{\epsilon}(\mathcal{A})$,
\begin{displaymath}
|d_{\mathcal{T}}(X,Y) - d_{\mathcal{A}}(\varphi(X),\varphi(Y))| = O(1)
\end{displaymath}
where $d_{\mathcal{A}}(\cdot,\cdot)$ is the sup-metric on \ $\mathbb{H}^j \times \mathcal{T}(S \setminus \mathcal{A})$.

\end{theorem}
We will call the connected components of ${S}_g \big\backslash \textbf{S}_{g\omega}$ the thick subsurfaces associated to $\omega^2$. Each subsurface $W \in {S}_g \big\backslash \textbf{S}_{g\omega}$ has a well-defined $\omega$-$\textit{diameter}$ (the smallest $W$ in its homotopy class with an $\omega$-geodesic boundary) which we denote by $W_{\omega}$.

\begin{theorem}[Rafi \cite{rafi2014thick}]
\label{theorem2.6}
For every essential closed curve $\gamma \in W$, 

\begin{displaymath}
\ell_X(\gamma) \overset{*}{\asymp} \frac{\ell_{\omega^2}(\gamma)}{W_{\omega}}.
\end{displaymath}
\end{theorem}

\end{subsection}
Suppose $\alpha \in \textbf{S}_g$. Let $\beta_1,\beta_2$ be two curves intersecting transversally with $\alpha.$ There exists an annulus $\tilde{{S}}_{g\alpha}$ covering $\alpha$. Let $\tilde{\beta}_1, \tilde{\beta}_2$ be lifts of $\beta_1$ and $\beta_2$, respectively, to $\tilde{S}_{g\alpha}$. The lifts $\tilde{\beta}_1$ and $\tilde{\beta}_2$ are not unique, but there exist (in their respective free homotopy classes) lifts of $\beta_1$ and $\beta_2$ so that $\tilde{\beta}_1$ and $\tilde{\beta}_2$ connect the boundary curves of $\tilde{S}_{g\alpha}$. Define

\begin{displaymath}
\mathrm{twist}_{\alpha}\left(\beta_1,\beta_2\right) = i\left(\tilde{\beta}_1,\tilde{\beta}_2\right).
\end{displaymath}

\begin{prop}[Minsky,\cite{minsky1996extremal}]
Let $\alpha \in \textbf{S}_g$. Let $\beta_1$ and $\beta_2$ be two curves intersecting transversally with $\alpha$. Let $\tilde{\beta}_1$ and $\tilde{\beta}_2$ be lifts connecting the boundary curves of \ $\tilde{S}_{g\alpha}$, and let \ $\overline{\beta}_1$ and $\overline{\beta}_2$ be another pair of respective lifts connecting the boundary curves of \ $\tilde{S}_{g\alpha}$. Then

\begin{equation}
\label{equation2.3}
\bigl| i\left(\tilde{\beta}_1,\tilde{\beta}_2\right) - i(\overline{\beta}_1,\overline{\beta}_2)\bigr| < O(1).
\end{equation}
\end{prop}
More generally, and as in $\cite{eskin2019counting}$, for any two structures on $\textbf{S}_g$, one can define a notion of twisting around $\alpha \in \textbf{S}_g$. We say two structures on $\textbf{S}_g$ admit "twisting data around $\alpha$" if one can always choose free homotopy classes of lifts for these structures, connecting the boundaries of the annular cover $\tilde{S}_{g\alpha}$ so that the intersection numbers between lifts of two curves intersecting transversally with $\alpha$ are well-defined up to Equation $\ref{equation2.3}$. To quantify twisting data as one travels along a geodesic in an orbit closure, we associate cohomology classes to core curves in a cylinder decomposition, as we will see in the next section.

\section{Twisting information}
\label{Section3}

We will specify a family of local charts covering $V$ and use the main result of $\cite{mirzakhani2017boundary}$ to show this collection is finite. The finiteness of this family is crucial to obtaining universal bounds on relations between intersection numbers of saddle connections, which are necessary to push the strategy in $\cite{eskin2019counting}$ through. 
Assume, now, that $(X,\omega^2) \in V$ is horizontally periodic. Let $\left\{C_1,...,C_r\right\}$ denote the horizontal cylinders and $\left\{\alpha_1,...,\alpha_r\right\}$ the horizontal saddle connections. The following is a standard observation.

\begin{prop}
\label{propposition3.1}
Assume $\left\{\alpha_1,...,\alpha_{r_0}\right\}$ is a maximal independent family of horizontal saddle connections in $H_1(X,Z(\omega),\mathbb{Z})$. Then, one can extend $\left\{\alpha_1,...,\alpha_{r_0}\right\}$ to a basis 

\begin{displaymath}
B:= \left\{\alpha_1,...,\alpha_{r_0},\beta_1,...,\beta_k\right\}
\end{displaymath}
where each $\beta_i$ is a saddle connection joining the left endpoint of a horizontal saddle connection in the bottom of the cylinder $C_i$ to the left endpoint of a horizontal saddle connection in the top of the cylinder $C_i$.
\end{prop}

\begin{lemma}[Nguyen,\cite{nguyen2012volumes}]
\label{lemma3.2}
Let $(X,\omega^2) \in V$ be horizontally periodic, let $\left\{C_1,...,C_r\right\}$ denote the horizontal cylinders, $\left\{\alpha_1,...,\alpha_r\right\}$ the horizontal saddle connections, and $\left\{\alpha_1,...,\alpha_{r_0}\right\}$ a maximal homologically indepedent set of horizontal saddle connections. Denote by $\Gamma$ the corresponding separatrix diagram. Assume $\left\{\alpha_1,...,\alpha_{r_0}\right\}$ is completed to a basis $B:=\left\{\alpha_1,...,\alpha_{r_0},\beta_1,...,\beta_k\right\}$ as in Proposition $\ref{propposition3.1}$. Define

\begin{displaymath}
V_{\Gamma}:= \left\{\left(l_1,...,l_{r_0},v_1,...,v_k\right) \in \left(\mathbb{R}_{>0}\right)^{r_0} \times \mathbb{C}^k\right\}.
\end{displaymath}
Let \ $\Phi: V_{\Gamma} \rightarrow \Omega\mathcal{T}_g(S)\left(p_1,...,p_k\right)$ be the map defined as follows: for \ $\left(l_1,...,l_{r_0},v_1,...,v_k\right)$, construct the cylinder $C_j$ from the parallelogram in \ $\mathbb{R}^2$ determined by the assignment of the length $l_i$ to the horizontal saddle connection $\alpha_i$ and the complex holonomy vector $v_i$ to each $\beta_i$. Glue the cylinders together according to the cylinder diagram $\Gamma$. Then, $\Phi$ is locally injective, and $\Phi^{-1}\left(V\right)$ is comprised of the intersection of \ $V_{\Gamma}$ with a family of real linear subspaces of \ $\mathbb{R}^{r_0} \times \mathbb{C}^k$ of dimension $\leq 2h-1$. 

\end{lemma}

In particular, if $S:= T_{X}V = S^{\mathbb{R}} \oplus iS^{\mathbb{R}}$, with $S^{\mathbb{R}} = S \cap H^1(X,Z(\omega),\mathbb{R})$, then the family of real linear subspaces in Lemma $\ref{lemma3.2}$ lies in $H^1(X,Z(\omega),\mathbb{C}) = H^1(X,Z(\omega),\mathbb{R} \oplus i\mathbb{R})$. If we denote each of these vector spaces by $V_{\alpha}$, then we have that 
\begin{displaymath}
\Phi(V_{\Gamma,\alpha}) :=  \Phi(V_{\Gamma} \cap V_{\alpha})
\end{displaymath}
covers $V$. A priori, there is no reason for this collection to be finite. In $\cite{nguyen2012volumes}$, finiteness is derived using algebraicity of $V$, based on analogies with Teichm{\"u}ller curves. Finiteness may also be obtained by appealing to the Cylinder Finiteness Theorem of $\cite{mirzakhani2017boundary}$. We will require a slightly weaker result.

\begin{definition}
Let $(X,\omega^2)$ be a completely periodic surface, and assume $\left\{C_1,...,C_r\right\}$ is a cylinder decomposition with core curves $\left\{\alpha_1,...,\alpha_r\right\}$. For each $\alpha_i$, define $I_{\alpha_i} \in H^1(X,Z(\omega),\mathbb{Z})$ to be the cohomology class which is zero on all relative homology classes homologous to a cycle disjoint from the interior of \ $C_i$, and equal to $1$ on any cycle connecting a zero of $\omega$ on the bottom of the cylinder $C_i$ to a zero on the top of the cylinder.
\end{definition}

\begin{theorem}[Mirzakhani-Wright,$\cite{mirzakhani2017boundary}$]
\label{theorem3.4}
Suppose $V \subset \Omega\mathcal{T}_g(S)(p_1,...,p_k)$ is locally cut out by real linear equations in period coordinates. Then, there exist two finite sets $S_1, S_2 \subset \mathbb{R}$, such that if \ $(X,\omega^2) \in V$, and $\mathcal{C} = \left\{C_1,...,C_r\right\}$ is a cylinder decomposition of $(X,\omega^2)$, then the ratio of circumferences of any two cylinders in $\mathcal{C}$ is in $S_1$. Furthermore, if $C_i$ has circumference $c_i$ and core curve $\alpha_i$, then the twist space for $\mathcal{C}$ at $(X,\omega^2)$ is a subspace of

\begin{displaymath}
\left\{\sum t_ic_iI_{\alpha_i}\right\}
\end{displaymath}
defined by linear equations on the $t_i$ with coefficients in $S_2$ that are satisfied when all $t_i = m_i$ are the moduli of cylinders $C_i$. In particular, the relations between the heights $h_i$ of the $C_i$ are restricted to a finite set coming from $S_2$.
\end{theorem}

\begin{corollary}
\label{corollary3.4.1}
The set $\left\{\Phi(V_{\Gamma,\alpha})\right\}$ covering $V$ may be chosen so that if 
\begin{displaymath}
V_{\Gamma,\alpha} = (l_1,...,l_{r_0},v_1,...,v_k),
\end{displaymath}
then there exists a finite set of possible linear relations among the holonomy vectors $v_i$.
\end{corollary}
\begin{proof}
First, note that the number of cylinder diagrams $\Gamma$ in a stratum $\Omega\mathcal{T}_g(S)\left(p_1,...,p_k\right)$ is finite. It thus suffices to prove that there can only exist a finite set of linear relations among the $v_i$ for a fixed separatrix diagram $\Gamma$. Note that a vector $v_i$ coming from a crossing saddle connection is determined by the circumference and height of the corresponding horizontal cylinder $C_i$. However, by Theorem $\ref{theorem3.4}$, there can only exist a finite set of relations between the circumferences and heights. 
\end{proof}

\section{Bounds on intersection numbers}
\label{Section4}

In this section, we use the dynamical properties of rank-1 $\mathrm{SL}\left(2,\mathbb{R}\right)$-orbit closures deduced in Section $\ref{Section3}$, to achieve relations between intersection numbers stricter than those obtained in $\cite{eskin2019counting}$, owing to the restrictions on relations between periods given by holonomy vectors. 

\begin{subsection}{Preliminaries}

Let $\alpha \in \textbf{S}_g$. If $\alpha$ is an arc in $(X,\omega^2)$, then either has a unique geodesic representative, or its geodesic representatives foliate a flat cylinder $C_{\alpha}$, bounded by saddle connections.

\begin{definition}[Eskin-Mirzakhani-Rafi,\cite{eskin2019counting}]
A curve is $\textit{short}$ in $\omega^2$ if \ $\mathrm{Ext}_X(\alpha) \leq \varepsilon_0$, for the \ $\varepsilon_0$ appearing in the statement of \ Theorem \ $\ref{theorem2.3}$. Denote the set of short curves in $\omega^2$ by ${\mathbf{S}_g}_{\omega}$.  If the interior of $C_{\alpha}$ is non-empty, we say $\alpha$ is a cylinder curve. Let $\tau > 0$. A curve \ ${\mathbf{S}_g}_{\omega}$  is a $\textit{large cylinder curve}$ if \ $\mathrm{Mod}(C_{\alpha}) \geq e^{-2\tau}$. Denote the set of large-cylinder curves by ${\textbf{S}}^{\geq \tau}_{g\omega}$ and define

\begin{displaymath}
{\textbf{S}}^{\leq \tau}_{g\omega} = \textbf{S}_{g\omega} \big\backslash {\textbf{S}}^{\geq \tau}_{g\omega}.
\end{displaymath}

\end{definition}

\begin{definition}[Eskin-Mirzakhani-Rafi,\cite{eskin2019counting}]
Let $(X,\omega^2)$ be an Abelian differential. For $\alpha$ a cylinder curve and $C_{\alpha}$ its cylinder of geodesic representatives, let $v_{\alpha}$ be an arc connecting the boundary curves of $C_{\alpha}$, so that $v_{\alpha}$ is perpendicular to $\alpha$. An $(\omega,\tau)$-regular triangulation $T$ of $\omega^2$ is a collection of disjoint saddle connections satisfying the following conditions:

\begin{enumerate}
    \item For all $\alpha \in \mathcal{S}^{\geq \tau}_q$, $T$ is disjoint from the interiors $\mathring{C}_{\alpha}$ of $C_{\alpha}$ and triangulates their complement:
    \begin{displaymath}
    \omega^2 \ \big\backslash \bigcup_{\alpha \in S^{\geq \tau}_q} \mathring{C}_{\alpha}.
    \end{displaymath}
    \item Any edge $w$ of $T$ that intersects a thick subsurface $W$ of $\omega^2$ satisfies $l_{\omega^2} \overset{*}{\prec} W$.
    \item For $\alpha \in {\textbf{S}}^{\leq\tau}_{g\omega}$ a cylinder curve, $v_{\alpha}$ intersects $T$ a uniformly bounded number of times.
\end{enumerate}

\end{definition}

Using Theorem $\ref{theorem2.6}$, it is proven in $\cite{eskin2019counting}$ that collections of disjoint saddle connections can always be extended to these types of triangulations.

\begin{lemma}[Eskin-Mirzakhani-Rafi,\cite{eskin2019counting}]
\label{lemma4.3}
For every $\tau$, there is  $\varepsilon_1(\tau)$ so that for $\varepsilon < \varepsilon_1(\tau)$, any subset $Q \subset Q_{\omega}(\varepsilon)$ consisting of disjoint saddle connections can be extended to a $(\omega,\tau)$-regular triangulation.
\end{lemma}

The following will also be useful.

\begin{lemma}[Eskin-Mirzakhani-Rafi,\cite{eskin2019counting}]
\label{lemma4.4}
Let $T$ be a $(\omega,\tau)$-regular triangulation and let $w_T$ be an edge of $T$. Let $s$ be the minimum of \ $W_{\omega}$ where $W$ is a thick subsurface of $\omega^2$ that intersects $w_T$. Let $w$ be any other saddle connection in $q$ satisfying $\mathrm{twist}(w,q) = O(1)$ for every curve $\alpha \in S_q^{\leq\tau}$. Then,

\begin{displaymath}
i(w_T,w) \overset{*}{\prec} \frac{\ell_{\omega}(w)}{s} + 1.
\end{displaymath}
\end{lemma}

\end{subsection}

\begin{theorem}
\label{theorem4.5}
For each $\alpha \in {\textbf{S}}^{\geq \tau}_{g\omega}$, let $\beta_{\alpha}$ be a saddle connection connecting the boundary curves of $C_{\alpha}$. Define

\begin{displaymath}
U_a = T_a \cup \bigcup_{\alpha \in {\textbf{S}}^{\leq \tau}_{g\omega}} \beta_{\alpha}.
\end{displaymath}
Let $\left\{w^i\right\}_i$ be the set of edges in $U_a$. Assume there exists a set of \ $\mathbb{R}$-linear equations between them

\[
\systeme*{\underset{i}{\sum}a_{1i}w^i=0,\underset{i}{\sum}a_{2i}w^i=0,...=0, ...=0, \underset{i}{\sum}a_{gi}w^i=0}
\]
Viewing $\left\{w^i\right\}_i$ as holonomy vectors in $\mathbb{C}^h$, let $w^i = \underset{j}{\sum}c^jw^j$, upon reducing to row echelon form. Then, for every triangle in $T_b$ with edges $w_1,w_2,w_3$, we have (with $q^j \in \mathbb{R})$

\begin{align*}
i(w^i,w_1) &= \underset{j}{\sum}c^ji(w^j,w_1)  \\ 
i(w^i,w_2) &= \underset{j}{\sum}c^ji(w^j,w_2)  \\ 
i(w^i,w_3) &= \underset{j}{\sum}c^ji(w^j,w_3). \\
\end{align*}
\end{theorem}

\begin{proof}
Since $(X,\omega^2)$ is a completely periodic surface, given saddle connections $\alpha$ and $\beta$ with slopes $s$ and $s'$, we may find (by Corollary $\ref{corollary2.1})$, for $s$ (respectively, $s'$), a cylinder decomposition of $(X,\omega^2)$ all of whose core curves have slope $s$ (respectively, $s'$). Let $\left\{C_1,..,C_r\right\}$ be such a cylinder decomposition in the direction $s$. Denote by $\left\{\alpha_1,...,\alpha_r\right\}$ the set of core curves of these cylinders. We make the following basic observation: if $\theta$ is the angle between $s$ and $s'$, then for each $i$

\begin{align}
\label{Eq3.1}
l(C_i \cap \beta)\mathrm{sin}(\theta) &= \mathrm{height}(C_i)i(\beta,\alpha_i) \\  
\implies i(\beta,\alpha_i)  &= \frac{l(C_i\cap \beta) \mathrm{sin}(\theta)}{\mathrm{height}(C_i)}.
\end{align}
It suffices to consider the case of $w_1$. Let $w^i = \underset{j}{\sum}c^jw^j$, with notation as in the theorem statement. Denote by $s_1$ the slope of $w_1$, and by $s_i$ and $s_j$ the slopes of $w^i$ and of each $w^j$, respectively. For $\left\{C_1^l\right\}$ a cylinder decomposition in the direction of $s_1$, Equation $\ref{Eq3.1}$ gives us

\begin{align}
i(w^i,(w_1)_l) &= \frac{l(C_1^l \cap w^i)\mathrm{sin}(\theta)}{\mathrm{height}(C_1^l)}  \label{eq.5.2}\\
i(w^j,(w_1)_l) &= \frac{l(C_1^l \cap w^i)\mathrm{sin}(\theta_j')}{\mathrm{height}(C_1^l)}  \label{eq.5.3}.
\end{align}
$(w_1)_l$ is the core curve of $C_1^{l}$, and each of these is parallel to $w_1$. From equations $\ref{eq.5.2}$ and $\ref{eq.5.3}$, respectively, we obtain

\begin{align}
\label{equation4.5}
    \underset{l}\sum l(C_1^l \cap w^i) &= \notag
    l(w^i) \\ &=  \frac{1}{\mathrm{sin}(\theta)}\underset{l}\sum i(w^i,(w_1)_l)\mathrm{height}(C_1^l) 
\end{align}
and 

\begin{align}
\label{equation4.6}
    l(w^j) = \frac{1}{\mathrm{sin}(\theta_j')}\underset{l}\sum i(w^j,(w_1)_l)\mathrm{height}(C_1^l)
\end{align}
where $\theta$ is the angle between $w^i$ and $w_1$, and $\theta_j'$ is the angle between $w^j$ and $w_1$. Since each of $(w_1)_l$ is a multiple of $w_1$ in homology, and intersection number is linear, equations $\ref{equation4.5}$ and $\ref{equation4.6}$ become, respectively

\begin{align}
\label{equation4.7}
    l(w^i) = \frac{1}{\mathrm{sin}(\theta)}\underset{l}\sum d_li(w^i,w_1)\mathrm{height}(C_1^l) 
\end{align}
and

\begin{align}
\label{equation4.8}
    l(w^j) = \frac{1}{\mathrm{sin}(\theta_j')}\underset{l}\sum d_li(w^j,w_1)\mathrm{height}(C_1^l) 
\end{align}
where $w_1 = d_l(w_1)_l$. Denote by $\mathcal{C}$ the sum $\underset{l}\sum d_li(w^i,w_1)\mathrm{height}(C_1^l)$. Then, we can write

\begin{align}
\label{equation4.9}
    l(w^i)\mathrm{sin}(\theta) = \mathcal{C}i(w^i,w_1)
\end{align}
and 

\begin{align}
\label{equation4.10}
    l(w^j)\mathrm{sin}(\theta_j') = \mathcal{C}i(w^j,w_1).
\end{align}
By the theorem statement, $w^i = \underset{j}\sum c^jw^j$. Since $w^i = l(w^i)\mathrm{cos}(\theta) + l(w^i)i\mathrm{sin}(\theta)$ and $w^j = l(w^j)\mathrm{cos}(\theta) + l(w^j)i\mathrm{sin}(\theta)$, we have

\begin{align*}
    w^i &= \underset{j} \sum c^jw^j\\
    \implies l(w^i)\mathrm{sin}(\theta) &= \underset{j} \sum c^j l(w^j)\mathrm{sin}(\theta_j')
\end{align*}
so by equations $\ref{equation4.9}$ and $\ref{equation4.10}$ we obtain

\begin{equation}
i(w^i,w_1) = \underset{j}\sum c^j i(w^j,w_1),
\end{equation}
proving the statement.
\\ \\
\end{proof}

\begin{remark}

 We explain the heuristic behind this proof. Note that we could have abused notation by writing $i(w^i,w_1)$ for all summands in Equation $\ref{equation4.7}$ and $i(w^j,w_1)$ for all summands in Equation $\ref{equation4.8}$. This would be accounted for by an additive error of $O(1)$, since if a saddle connection with slope $s_1$ intersects the core curve of a cylinder with boundary curve $w^i$ (resp $w^j$), then the saddle connection with slope $s_1$ intersects the next core curve in the cylinder decomposition, and so on. The additive error is uniform because Theorem $\ref{theorem3.4}$ guarantees that there can only be a finite set of relations between the $\left\{\alpha_1,...,\alpha_r\right\}.$ Now, we can assume $\frac{\mathrm{sin}(\theta)}{\mathrm{sin}(\theta')}$ is finite, since otherwise $w^j$ and $w^i$ are parallel. By the assumptions of the theorem statement, we have

\begin{equation}
    w^i = \sum_k a_{kj}w^j
\end{equation}
whence, we have, by the triangle inequality

\begin{displaymath}
l(w^i) \leq \sum_k a_{kj}l(w^j)
\end{displaymath}
and in fact

\begin{displaymath}
\bigl|l(w^i) - \sum_k a_{kj}l(w^j)\bigr|\ = O(1).
\end{displaymath}
\end{remark}

Using a nearly identical argument, we obtain a related theorem. 

\begin{theorem}
\label{theorem4.6}
For each $\alpha \in {\textbf{S}}^{\geq \tau}_{g\omega}$, let $\beta_{\alpha}$ be a saddle connection connecting the boundary curves of $C_{\alpha}$. Define

\begin{displaymath}
U_a = T_a \cup \bigcup_{\alpha \in {\textbf{S}}^{\leq \tau}_{g\omega}} \beta_{\alpha}.
\end{displaymath}

Let $\left\{w^i\right\}_i$ be the set of edges in $U_a$. Assume there exists a set of \ $\mathbb{R}$-linear equations between them

\[
\systeme*{\underset{i}{\sum}a_{1i}w^i=0,\underset{i}{\sum}a_{2i}w^i=0,...=0, ...=0, \underset{i}{\sum}a_{gi}w^i=0}
\]

Viewing $\left\{w^i\right\}_i$ as holonomy vectors in \ $\mathbb{C}^h$, let $w^i = \underset{j}{\sum}c^jw^j$, upon reducing to row echelon form. Then, for any $\alpha_b \in \textbf{S}^{\geq \tau}_{g\omega_b}$ we have

\begin{displaymath}
i(w^i,\alpha_b) = \underset{j}{\sum}q^j i(w^j,\alpha_b) + O(1).
\end{displaymath}

\end{theorem}

\begin{corollary}
\label{corollary4.6.1}
For the set $\left\{\Phi(V_{\Gamma,\alpha})\right\}$ covering $V$ with

\begin{displaymath}
V_{\Gamma,\alpha} = (l_1,...,l_{r_0},v_1,...,v_k)
\end{displaymath}
there exists a finite number of relations between the $i(v_i,v)$, where $v$ is some fixed $v = v_j$, up to additive error. 
\end{corollary}
\begin{proof}
Corollary $\ref{corollary3.4.1}$ and Theorem $\ref{theorem4.5}$.
\end{proof}

Similarly

\begin{corollary}
For the set $\left\{\Phi(V_{\Gamma,\alpha})\right\}$ covering $V$ with

\begin{displaymath}
V_{\Gamma,\alpha} = (l_1,...,l_{r_0},v_1,...,v_k)
\end{displaymath}
there exists a finite number of relations between the $i(v_i,\alpha)$, where $v$ is some fixed simple closed curve $\alpha$, up to additive error. 
\end{corollary}

\begin{proof}
Corollary $\ref{corollary3.4.1}$ and Theorem $\ref{theorem4.6}$.
\end{proof}
For a fixed constant $r_0$, define $\mathcal{B}(V,X,\tau)$ to be the set of points along a Teichm{\"u}ller geodesic:

\begin{displaymath}
G:[a,b] \rightarrow \Omega_1\mathcal{T}_g(S)\left(p_1,...,p_k\right), \quad G(t) = (X_t,{\omega}^{2}_t),
\end{displaymath}
such that

\begin{displaymath}
d_{\mathcal{T}}(X,X_a) \leq r_0 \quad d_{\mathcal{T}}(X_b,Z) \leq r_0, \quad b-a \leq \tau
\end{displaymath}
and 

\begin{displaymath}
(X_t,{\omega}^{2}_t) \in \overline{\Gamma_{\omega}\cdot\mathrm{SL}(2,\mathbb{R})\cdot (X,\omega^2)}.
\end{displaymath}
Essentially, $\mathcal{B}(V,X,\tau)$ is the ball of radius $\tau$ centered at $(X,q)$, except on is only allowed to travel in the direction of the closure of the Teichm{\"u}ller disk. Recall $\mathrm{Mod}(S_g)$ denotes the mapping class group of $S_g$. We define subsets of this ball which will allow us to state the main theorem in this paper.

\begin{definition}[Eskin-Mirzakhani-Rafi,\cite{eskin2019counting}]
\label{definition4.7}
Let $\mathcal{B}_j(V,X,\tau) \subset \mathcal{B}(V,X,\tau)$ be the set of points \ $Z \in \Omega_1\mathcal{T}_g(S)\left(p_1,...,p_k\right)$ so that, for the associated quadratic differentials $\omega^2_a$, $\omega^2_b$ there exists an $(\omega^2_a,\tau)$-regular triangulation $T_{\omega_a}$ and an $(\omega^2_b,\tau)$-regular triangulation $T_{\omega_b}$ that have $j$ common homologically independent saddle connections. Furthermore, let

\begin{displaymath}
\mathcal{B}(V,X,Y,\tau) = \mathcal{B}(V,X,\tau) \cap (\mathrm{Mod}(S_g)\cdot Y)
\end{displaymath}
and

\begin{displaymath}
\mathcal{B}_j(V,X,Y,\tau) = \mathcal{B}_j(V,X,\tau)\cap(\mathrm{Mod}(S_g) \cdot Y).
\end{displaymath}
\end{definition}

For the $\varepsilon_0$ in Theorem $\ref{theorem2.3}$, recall that we say a curve $\alpha$ is $\textit{short}$ on $X$ if $\mathrm{Ext}_X(\alpha) \leq \varepsilon_0$. Let ${S_g}_{X}$ be the set of short curves on $X$. Define $\mathcal{G}:\mathcal{T}_g(S) \rightarrow \mathbb{R}_+$ by

\begin{equation}
\label{equation4.13}
\mathcal{G}(X) =  1 + \underset{\alpha \in {S_g}_{X}}{\prod} \frac{1}{\sqrt{\mathrm{Ext}_X(\alpha)}}.
\end{equation}

\begin{prop}
Let $\left\{w_b\right\}_i$ be a collection of $j$ homologically independent edges in $T_b$. Then, $\mathrm{dim}_{\mathbb{R}}\langle{\mathcal{R}\rangle} = \mathrm{dim}_{\mathbb{R}}\langle{\left\{w_b\right\}_i\rangle}$.
\end{prop}

We will generalize the following Theorem proven in $\cite{eskin2019counting}$.

\begin{theorem}[Theorem 5.1,\cite{eskin2019counting}]
\label{theorem4.9}
Consider the stratum $\Omega_1\mathcal{T}_g(S)\left(p_1,...,p_k\right)$. For $X,Y \in \Omega_1\mathcal{T}_g(S)\left(p_1,...,p_k\right)$

\begin{displaymath}
|\mathcal{B}_j(\Omega_1\mathcal{T}_g(S)\left(p_1,...,p_k\right),X,Y,\tau)| \overset{*}{\prec} \tau^{|{S_g}_{X}| + |{S_g}_{Y}|}e^{({h-j})\tau}\mathcal{G}(X)\mathcal{G}(Y).
\end{displaymath}
where $h$ is the dimension of the stratum.
\end{theorem}

\begin{definition}
A marking $\textbf{M} = \textbf{M}\left\{\textbf{S}_g,\left\{E(\alpha)\right\}, T\right\}$ for $S_g$ consists of all of the following

\begin{enumerate}
    \item a collection of simple closed curves $\textbf{S}_g$ whose free homotopy classes can be realized pairwise disjointly,
    \item a length $E(\alpha)$ associated to each of these curves $\alpha \in \textbf{S}_g$,
    \item a homotopy class of a partial triangulation $T$ with vertex set $\sigma$, so that the core curve of any annulus in the complement of $T$ is in $\textbf{S}_g$.
\end{enumerate}
\end{definition}

\begin{definition}
A marking $\textbf{M} = \textbf{M}\left\{\textbf{S}_g,\left\{E(\alpha)\right\}, T\right\}$, is said to be a marking with $\textit{twisting datum}$ if all of the following conditions are satisfied

\begin{enumerate}
    \item For any curve $\alpha \in \mathbf{S}_g$, and any two curves $\beta_1$, $\beta_2$, with $\beta_2 \in T$, the lifts $\tilde{\beta}_1$ and $\tilde{\beta}_2$ to the annular cover associated to $\alpha$ satisfy $i(\tilde{\beta}_1,\tilde{\beta}_2) < O(1)$ (by Equation $\ref{equation2.3}$, this is well-defined).
    \item There exists an $X$ with topological type $S_g$ so that \ $\mathbf{S}_g = {S_g}_{X}$, the set of \ $\varepsilon_0$-short curves in $X$.
    \item $E(\alpha) = \mathrm{Ext}_{\alpha}(X)$.
    \item For $\alpha \in {\mathbf{S}}_g^{c}$, the set of simple closed curves disjoint from $T$, and $\beta_1 \in \mathbf{S}_g$, $\beta_2$ any simple closed curve in $X$, let $\tilde{\beta}_1$ and $\tilde{\beta}_2$ be the lifts associated to the annular cover of $\alpha$. Then, $i(\tilde{\beta}_1,\tilde{\beta}_2) < O(1)$. (This is well-defined by Equation $\ref{equation2.3}$).
    \item For each simple closed curve $\gamma$ disjoint from ${S_g}_X$, the $\textit{minimum}$ number of arcs of $T$ that can appear in a representative of $\gamma$ is equal to the length of $\gamma$ in $X$, up to a uniform constant factor depending only on the genus $g$.
    \item For $\alpha$ a simple closed curve not in ${\textbf{S}}_g^{c}$, $\beta_1 \in \textbf{S}_g$, $\beta_2$ any simple closed curve in $X$, let $\tilde{\beta}_1$, $\tilde{\beta}_2$ be the lifts to the annular cover over $\alpha$. Then, $i(\tilde{\beta}_1,\tilde{\beta}_2) < O(\tau)$, which is well-defined again by Equation $\ref{equation2.3}$.
\end{enumerate}
\end{definition}

In the sequel, when $\alpha$ is a simple closed curve and we are interested in a uniform bound on the intersection number in the annular lift of two tranversally intersecting curves $\beta_1$ and $\beta_2$, where $\beta_1$ and $\beta_2$ are in some geometric structure $D_1$ and $D_2$, respectively, we will denote 

\begin{displaymath}
\mathrm{twist}_{\alpha}(D_1,D_2):= i(\tilde{\beta}_1,\tilde{\beta}_2).
\end{displaymath}

\begin{definition}
Consider the markings $\textbf{M} = \textbf{M}\left\{\textbf{S}_g,\left\{E(\alpha)\right\},T\right\}^{twist}$ and $\textbf{M}'\left\{\textbf{S}_g,\left\{E'(\alpha')\right\},T'\right\}^{twist'}$ with twisting data. $T$ and $T'$ have the same vertex set $\Sigma$. For every $\alpha  \in {\textbf{S}^c_g}$, let $\beta_{\alpha}$ be an arc with end point in $\Sigma$ and disjoint from $T$, so that $\beta_{\alpha}$ crosses $\alpha$ and $T \cup \beta_{\alpha}$ has bounded twisting around $\alpha$. Denote

\begin{displaymath}
U = T \cup \underset{\alpha \in \textbf{S}^c_g}{\bigcup}{\beta_{\alpha}}.
\end{displaymath}
Let $\mathbb{R}[U]$ be the vector space of formal sums with real coefficients of edges in $U$. Let $\mathcal{R}^{orb}$ be a subset of \ $\mathbb{R}[U]$. "$orb$" will be a placeholder for "orbit relations" coming from the linear equations defining the closed orbit. Let $\mathcal{R}$ be another subset of $\mathbb{R}[U]$. Define the set 

\begin{displaymath}
M^{twist}_{\mathcal{R}}(\textbf{M},\textbf{M'},\tau)
\end{displaymath}
to be the set of markings $\tilde{\textbf{M}} = \left\{\tilde{\textbf{S}_g},\left\{\tilde{E}(\tilde{\alpha}\right\},\tilde{T}\right\}$ satisfying

\begin{enumerate}
  \item $\tilde{\textbf{M}}$ is a homeomorphic image of $\textbf{M}'$.
  \item For every $\tilde{\alpha} \in \tilde{\textbf{S}^c_g}$, which is the image of $\alpha^{'} \in {\textbf{S}^c}_g^{'}$, $\tilde{E}(\tilde{\alpha}) = {E^{'}}(\alpha')$.
  \item For every $\sum q_w w \in \mathcal{R}$, we have the relations
  
  \begin{displaymath}
      \underset{w \in U}\sum q_w i(w,\tilde{w}) = O(1),  \hspace{1cm} \underset{w \in U} \sum q_w i(w,\tilde{\alpha})=O(1), \\
        \end{displaymath}
      and for every $\sum q^{orb}_w w \in {\mathcal{R}}^{orb}$ we have
      
      \begin{displaymath}
      \underset{w \in U}\sum q_w i(w,\tilde{w}) = O(1),  \hspace{1cm} \underset{w \in U} \sum q_w i(w,\tilde{\alpha})=O(1) \\
      \end{displaymath}
      
      valid for every arc $\tilde{w}\in \tilde{T}$ and every $\tilde{\alpha} \in \tilde{\textbf{S}^c_g}$.
      \item Given $\alpha \in {\textbf{S}^c_g}$, $\tilde{\alpha} \in \tilde{\textbf{S}^c_g}$, $w \in T$ and $\tilde{w} \in \tilde{T}$ the following bounds are satisfied
      
      \begin{align}
          i(w,\tilde{w})&\overset{*}{\prec}e^{\tau} \\
          \mathrm{twist}_{\tilde{\alpha}}\left(\textbf{M},{\tilde{\textbf{M}}}\right)\sqrt{E(\tilde{\alpha})}i(\tilde{\alpha},T)&\overset{*}{\prec}e^{\tau}\\
          \mathrm{twist}_{{\alpha}}\left(\textbf{M},{\tilde{\textbf{M}}}\right)\sqrt{E({\alpha})}i({\alpha},\tilde{T})&\overset{*}{\prec}e^{\tau}\\
          i(\alpha,\tilde{\alpha})\mathrm{twist}_{\alpha}\left(\textbf{M},\tilde{\textbf{M}}\right)\sqrt{\tilde{E}(\tilde{\alpha})E(\alpha)} &\overset{*}{\prec}e^{\tau}\\
          (\mathrm{for}\; \alpha = \tilde{\alpha}) \hspace{0.2cm} \mathrm{twist}_{\alpha}\left(\textbf{M},\tilde{\textbf{M}}\right)\sqrt{E(\alpha)\tilde{E}(\tilde{\alpha})} &\overset{*}{\prec}e^{\tau}.
      \end{align}
\end{enumerate}

\end{definition}

\begin{lemma}
\label{lemma4.13}
For $\tilde{\textbf{M}} \in M^{twist}_{\mathcal{R}}\left(\textbf{M},\textbf{M'},\tau\right)$, consider the weighted graphs

\begin{displaymath}
\tilde{W} = \tilde{T} + \underset{\tilde{\alpha} \in \tilde{S}^c_g}\sum m \left(\tilde{\alpha},\tilde{M}\right)\tilde{\alpha}
\end{displaymath}
where the weight $m\left(\tilde{\alpha},\tilde{M}\right) \in \mathbb{N}$ is the $\textit{floor}$ of \ $\mathrm{twist}_{\tilde{\alpha}}\left(\textbf{M},\tilde{\textbf{M}}\right)\sqrt{\tilde{E}(\tilde{\alpha})}$. Define \ $W$ to be the set of weighted graphs

\begin{displaymath}
W = \left\{ \tilde{W} \bigl| \tilde{\textbf{M}} \in M_{\mathcal{R}}^{twist} \left(\textbf{M},\textbf{M'},\tau\right)\right\}.
\end{displaymath}
Consider the set $U_{0} \subset U$ that forms a basis for \ $\mathbb{R}[U]\big/\langle{\mathcal{R}^{orb}, \mathcal{R} \rangle}$. Then, there exist choices for $\mathcal{R}$ and $\mathcal{R}^{orb}$, which satisfy the following. The map 

\begin{displaymath}
I: W \rightarrow \mathbb{N}^{h_{\mathcal{R}^{orb},\mathcal{R}}}, \hspace{1cm} \tilde{W} \rightarrow \left(i\left(\tilde{W},w\right)\right)_{w \in U_{0}}
\end{displaymath}
is finite-to-one, where

\begin{displaymath}
i\left(\tilde{W},w\right):= \underset{\tilde{w} \in \tilde{T}}{\sum}i\left(\tilde{w},w\right) + \underset{\tilde{\alpha} \in \tilde{\textbf{S}}^c_g}{\sum}m\left(\tilde{\alpha},\tilde{\textbf{M}}\right)i\left(\tilde{\alpha},w\right).
\end{displaymath}

\end{lemma}

\begin{proof}
Let $1\cdot u \in \mathbb{R}[U]$. Then, there exist constants $q_w$ and $c_{R}$ so that:

\begin{equation}
\label{equation4.13}
1\cdot u = \underset{w\in U_0}{\sum}q_{w}w \ + \underset{R \in \langle{\mathcal{R}^{orb},\mathcal{R}\rangle}}{\sum}c_R R.
\end{equation}
By Proposition $\ref{propposition3.1}$, for a choice of $X$ with topological type $S_g$ and

\begin{displaymath}
\beta:= \left\{\alpha_1,...,\alpha_{r_0},\beta_1,...,\beta_k\right\}
\end{displaymath}
is an integral basis for $H_1(X,Z(\omega),\mathbb{Z})$ with $\left\{\alpha_i\right\}$ maximally independent and horizontal, and $\left\{\beta_i\right\}$ crossing saddle connections. The arcs in $U$ generate $H_1(X,Z(\omega),\mathbb{R})$, so we know that any $\mathbb{R}$-linear combination involving $\left\{\alpha_i\right\}$ and $\left\{\beta_i\right\}$ can be written as an $\mathbb{R}$-linear expression involving $w \in U$. In particular, the right-hand term of Equation $\ref{equation4.13}$ can be written with coefficients chosen by this process. Choose one of the linear relations from Corollary $\ref{corollary4.6.1}$ and  the linear relations in Lemma 4.7 from $\cite{eskin2019counting}$. Then, we know that the $c_R$ can be chosen to be uniformly bounded. Therefore, there are only a finite number of possible values $i(1\cdot u,\tilde{w})$ can take. 
\end{proof}
Lemma $\ref{lemma4.13}$ is "Step 1" of Lemma 5.8 in $\cite{eskin2019counting}$. "Step 2" does not require modification to adapt to our setting, and we record it here as a lemma.

\begin{lemma}
\label{lemma4.14}
$|W| \overset{*}{\prec} |I(W)| \leq e^{|U_0|\tau}\underset{\alpha \in \textbf{S}^c_g}{\prod}\frac{1}{\sqrt{E(\alpha})}$.
\end{lemma}

Combining Lemmas $\ref{lemma4.13}$ and $\ref{lemma4.14}$, we have the following fundamental lemma and its corollary.

\begin{lemma}
\label{lemma4.15}
Let ${h}_{\mathcal{R}^{orb},\mathcal{R}} = \mathrm{dim}_{\mathbb{R}}\left(\mathbb{R}[U]\big/\langle{\mathcal{R}^{orb}, \mathcal{R} \rangle}\right).$ Then,

\begin{displaymath}
|M_{\mathcal{R}}^{twist}\left(\textbf{M},\textbf{M}',\tau\right)| \overset{*}{\prec} e^{h_{\mathcal{R}^{orb},\mathcal{R}}\tau}\underset{\alpha \in \textbf{S}^c_g}{\prod}\frac{1}{\sqrt{E(\alpha})} \underset{\alpha'\in {\textbf{S}'}^c_g}{\prod} \frac{1}{\sqrt{E'(\alpha')}}.
\end{displaymath}
\end{lemma}

By repeated applications of Lemma $\ref{lemma4.4}$ (see $\cite{eskin2019counting}$, section 4.1), we have

\begin{corollary}
\label{corollary4.14.1}
\begin{displaymath}
|\mathcal{B}_j(V,X,Y,\tau)| \overset{*}{\prec}\tau^{|S_X|+|S_Y|}e^{(h-j)\tau}\mathcal{G}(X)\mathcal{G}(Y).
\end{displaymath}
\end{corollary}

$\mathcal{R}^{orb}$ was chosen so that the following holds.

\begin{prop}
$\mathrm{dim}_{\mathbb{R}} \left(\mathbb{R}[U]\big/\langle{\mathcal{R}^{orb},\mathcal{R}\rangle}\right) = h-j$.
\end{prop}

\begin{section}{Random walks}
\label{Section5}

In this section, we sketch the integration of the results of Section $\ref{Section4}$ into the random walk argument of $\cite{eskin2019counting}$. The proofs largely go through verbatim - we provide an overall outline of the argument, and proofs and comments where necessary. \\

We define $\mathcal{Q}_{j,\epsilon}\left(p_1,...,p_k\right)$ to be the set of  $(X,\omega^2)  \in \Omega_1\mathcal{T}_g(S)\left(p_1,...,p_k\right)$ such that $(X,\omega^2)$ has $\geq j$ homologically  independent saddle connections. Throughout this section, we fix a stratum, so we will not specify $\left(p_1,...,p_k\right)$.

\begin{definition}
Let $p:\mathcal{T}_g(S) \rightarrow \mathcal{M}_g(S)$ denote the natural projection. Let $\mathcal{N} \subset \mathcal{M}_g(S)$ be a net in $\mathcal{M}_g(S)$ where the distance between any two net points is always greater than a fixed $c > 0$ and any point in $\mathcal{M}_g(S)$ always comes within a distance $2c > 0$ of the net. We define

\begin{enumerate}
    \item $\mathcal{B}(\mathcal{Q}_{j,\varepsilon}, X,\tau) = \mathcal{B}(V,X,\tau) \cap \mathcal{Q}_{j,\epsilon}$,
    \item $\mathcal{N}(X,\tau) = p\left(\mathcal{B}(X,\tau)\right) \cap \mathcal{N}$,
    \item $\tilde{\mathcal{N}} = p^{-1}(\mathcal{N})$,
    \item $\tilde{\mathcal{N}}(V,X,\tau) = \mathcal{B}(V,X,\tau) \cap \tilde{\mathcal{N}}$,
    \item $\tilde{\mathcal{N}}(V,X,Y,\tau) = \mathcal{B}(V,X,Y,\tau) \cap \tilde{\mathcal{N}}$,
    \item $\tilde{\mathcal{N}}(\mathcal{Q}_{j,\epsilon},X,\tau) = \mathcal{B}(\mathcal{Q}_{j,\epsilon}, X, \tau) \cap \tilde{\mathcal{N}},$
    \item $\tilde{\mathcal{N}}(\mathcal{Q}_{j,\epsilon}, X, Y, \tau) = \mathcal{B}(\mathcal{Q}_{j,\epsilon}, X, Y,\tau) \cap \tilde{\mathcal{N}}$.
\end{enumerate}
\end{definition}

Using Lemma $\ref{lemma4.4}$ and the choice of $\varepsilon_1(\tau)$ from Lemma $\ref{lemma4.4}$, we have for any $\tau > 0$, an $\varepsilon_2(\tau) < \varepsilon_1(\tau)$, so that for all $X,Y \in \mathcal{T}_g(S)$:

\begin{equation}
\label{equation5.1}
    \mathcal{B}(\mathcal{Q}_{j,\epsilon},X,\tau) \subset \mathcal{B}_j(V,X,\tau)
\end{equation}
and

\begin{equation}
    \label{equation5.2}
    \mathcal{B}(\mathcal{Q}_{j,\epsilon},X,Y,\tau) \subset \mathcal{B}_j(V,X,\tau).
\end{equation}

\begin{lemma}
\label{Lemma5.2}
There exists a constant $c_0 > 0$ such that for any $c>c_0$, and net $\mathcal{N}$ as above, we have

\begin{displaymath}
\Big|p\left(B(X,\tau)\right) \cap \mathcal{N}\Big| \overset{*}{\prec} \tau^{3g-3}.
\end{displaymath}
\end{lemma}

The crucial application of Corollary $\ref{corollary4.14.1}$ is to obtain the following fundamental inequality. The proof goes through verbatim.

\begin{lemma}[cf. Proposition 6.3 \cite{eskin2019counting}]
Let $\mathcal{G}$ be the function defined in Equation $\ref{equation4.7}$. Define the averaging function

\begin{equation}
    \left(A^{\tau}_{j,\varepsilon}\mathcal{G}\right)(X) = e^{-2\tau}\underset{Z \in \tilde{\mathcal{N}} \left(\mathcal{Q}_{j,\varepsilon}, X, \tau \right)}{\sum} \mathcal{G}(Z)
\end{equation}
where $\left(A^{\tau}_{j,\varepsilon}\mathcal{G}\right)$ is viewed as a function from $\mathcal{T}_g(S)$ to \ $\mathbb{R}$. Then, given $\tau >0$ and $\epsilon$ small enough (< $\epsilon_2(\tau)$ from Equations \ $\ref{equation5.1}$ and $\ref{equation5.2}$ )

\begin{equation}
    \left(A^{\tau}_{j,\epsilon}\mathcal{G}\right)(X) \overset{*}{\prec} \tau^{m} e^{-j\tau}\mathcal{G}(X)
\end{equation}
where $m$ depends only on the genus of \ $S_g$.
\end{lemma}

\begin{definition}[\textbf{Random walk}]
Suppose $R >> \tau$ and let $n$ be an integer part of \ $R / \tau$. A $\textit{trajectory of a random walk}$ is a map

\begin{displaymath}
\lambda:\left\{0,n\right\} \rightarrow \tilde{\mathcal{N}}
\end{displaymath}
satisfying the following properties

\begin{enumerate}
    \item for all \ $0 < k \leq n$, $d_{\mathcal{T}}\left(\lambda(k),\lambda(k-1)\right) \leq \tau$,
    \item $d_{\mathcal{T}}(\lambda_0,X) = O(1)$,
    \item for $1 \leq k \leq n, \lambda(k) \in \tilde{\mathcal{N}}\left(V,\lambda(k), \tau\right)$,
    \item $\Big|\left\{k \ \big|1\leq k \leq n\right\}, \lambda(k) \in \mathcal{B}\left(\mathcal{Q}_{j,\epsilon}, \lambda(k-1), \tilde{\mathcal{N}},\tau \right) \Big| \geq \theta \cdot n$.
\end{enumerate}

We denote the set of all such random walk trajectories by $\mathcal{P}_{\theta,\tau}\left(\mathcal{Q}_{j,\epsilon},X,R\right)$. Let $X,Y \in \mathcal{T}_g(S)$, and define 

\begin{displaymath}
\mathcal{P}_{\theta,\tau}\left(\mathcal{Q}_{j,\epsilon}, X, Y, R \right):= \left\{ \lambda \in \mathcal{P}_{\theta,\tau}\left(\mathcal{Q}_{j,\epsilon},X,R\right) \big| d_{\mathcal{T}}\left(p(Y),p(\lambda(n))\right) = O(1), \ d_{\mathcal{T}}\left(p\left(\lambda(0)\right),p\left(\lambda(n)\right)\right)\right\}.
\end{displaymath}

\end{definition}

We recall a few lemmas.

\begin{lemma}[cf. Lemma 6.4 \cite{eskin2019counting}]
For any $\delta_0 > 0$, there is $\tau_0 > 0$ so that for $\tau > \tau_0$, $0 \leq \theta \leq 1$, and $\epsilon$ small enough we have

\begin{equation}
\label{equation5.3}
\Big| \mathcal{P}_{\theta,\tau}\left(\mathcal{Q}_{j,\varepsilon},X,X,R\right)\Big| \overset{*}{\prec} e^{\left(h-j\theta + \delta_0\right)}.
\end{equation}

\end{lemma}

\begin{lemma}[cf. Lemma 6.5 \cite{eskin2019counting}]
For any $\delta_1 > 0$, there exists $\tau_1$ so that for $\tau > \tau_1$, $X \in \mathcal{T}_g(S)$, and any sufficiently large $R$ (depending only on $\delta_1$ and $\tau$), we have

\begin{equation}
\label{equation5.4}
\mathcal{N}_{\theta}\left(V_{j,\varepsilon},p(X),(1-\delta_1)R\right) \overset{*}{\prec} \Big|\mathcal{P}_{\theta,\tau}\left(Q_{j,\varepsilon},X,R\right)\Big|.
\end{equation}
\end{lemma}

To tie it all together, one needs a classical estimate originally due to Veech.

\begin{lemma}[Veech]
\label{lemma5.7}
Suppose $\gamma$ is a closed geodesic of length $\leq R$ in $\mathcal{M}_g(S)$. Then, for any $X \in \gamma$, the extremal length of the shortest simple closed curve on $X$, $\mathrm{Ext}_{X}(\alpha)$, satisfies

\begin{displaymath}
\mathrm{Ext}_X(\alpha) \overset{*}{\succ} e^{-\left(6g-4\right)R}.
\end{displaymath}
\end{lemma}

\begin{proof}[Proof of Theorem \ref{Theorem1.1}]
Let $\delta > 0$. Choose $\delta_0,\delta_1 \leq \delta/3$, and $\tau \geq \mathrm{max}\left\{\tau_0,\tau_1\right\}$.  Let $R$ be large enough so that Equations $\ref{equation5.3}$ and $\ref{equation5.4}$ hold. This gives us

\begin{displaymath}
\mathcal{N}_{\theta}\left(V_{j,\varepsilon},p(X),R\right) \overset{*}{\prec} e^{\left(h-j\theta +2\delta/3\right)R}.
\end{displaymath}
By the definition of the net $\mathcal{N}$ above, we have

\begin{displaymath}
\mathcal{N}_{\theta}\left(V_{j,\varepsilon},R\right) \leq \sum_{p(X) \in \mathcal{N}} \mathcal{N}_{\theta}\left(Q_{j,\varepsilon},p(X),R\right).
\end{displaymath}
Now, by Lemmas $\ref{lemma5.7}$ and $\ref{Lemma5.2}$, the number of points in the net is a polynomial in $R$, so for $R$ large enough, this polynomial is less than $e^{\delta R/3}$, and the conclusion of the theorem holds. 

\end{proof}

\end{section}

\bibliographystyle{plain}
\bibliography{bibliography.bib}

\begin{thebibliography}{10}

\bibitem{eskin2008counting}
Alex Eskin and Maryam Mirzakhani.
\newblock Counting closed geodesics in moduli space.
\newblock {\em arXiv preprint arXiv:0811.2362}, 2008.

\bibitem{eskin2018invariant}
Alex Eskin and Maryam Mirzakhani.
\newblock Invariant and stationary measures for the action on moduli space.
\newblock {\em Publications math{\'e}matiques de l'IH{\'E}S}, 127(1):95--324,
  2018.

\bibitem{eskin2015isolation}
Alex Eskin, Maryam Mirzakhani, and Amir Mohammadi.
\newblock Isolation, equidistribution, and orbit closures for the sl (2, r)
  action on moduli space.
\newblock {\em Annals of Mathematics}, pages 673--721, 2015.

\bibitem{eskin2019counting}
Alex Eskin, Maryam Mirzakhani, and Kasra Rafi.
\newblock Counting closed geodesics in strata.
\newblock {\em Inventiones mathematicae}, 215(2):535--607, 2019.

\bibitem{kerckhoff1978asymptotic}
Steven~Paul Kerckhoff.
\newblock {\em The asymptotic geometry of Teichm{\"u}ller space.}
\newblock Princeton University, 1978.

\bibitem{margulis2003some}
Gregori~A Margulis.
\newblock {\em On some aspects of the theory of Anosov flows}.
\newblock PhD thesis, Ph. D. Thesis, 1970, Springer, 2003.

\bibitem{mcmullen2003billiards}
Curtis McMullen.
\newblock Billiards and teichm{\"u}ller curves on hilbert modular surfaces.
\newblock {\em Journal of the American Mathematical Society}, 16(4):857--885,
  2003.

\bibitem{minsky1996extremal}
Yair~N Minsky.
\newblock Extremal length estimates and product regions in teichm{\"u}ller
  space.
\newblock {\em Duke Mathematical Journal}, 83(2):249--286, 1996.

\bibitem{mirzakhani2008growth}
Maryam Mirzakhani.
\newblock Growth of the number of simple closed geodesies on hyperbolic
  surfaces.
\newblock {\em Annals of Mathematics}, pages 97--125, 2008.

\bibitem{mirzakhani2017boundary}
Maryam Mirzakhani and Alex Wright.
\newblock The boundary of an affine invariant submanifold.
\newblock {\em Inventiones mathematicae}, 209(3):927--984, 2017.

\bibitem{nguyen2012volumes}
DM~Nguyen.
\newblock Volumes of the sets of translation surfaces with small saddle
  connections.
\newblock {\em arXiv preprint arXiv:1211.7314}, 2012.

\bibitem{rafi2014thick}
Kasra Rafi.
\newblock Thick-thin decomposition for quadratic differentials.
\newblock {\em arXiv preprint arXiv:1407.4789}, 2014.

\bibitem{smillie2004minimal}
John Smillie and Barak Weiss.
\newblock Minimal sets for flows on moduli space.
\newblock {\em Israel Journal of Mathematics}, 142(1):249--260, 2004.

\bibitem{veech1989teichmuller}
William~A Veech.
\newblock Teichm{\"u}ller curves in moduli space, eisenstein series and an
  application to triangular billiards.
\newblock {\em Inventiones mathematicae}, 97(3):553--583, 1989.

\bibitem{wright2015cylinder}
Alex Wright.
\newblock Cylinder deformations in orbit closures of translation surfaces.
\newblock {\em Geometry \& Topology}, 19(1):413--438, 2015.

\bibitem{wright2016rational}
Alex Wright.
\newblock From rational billiards to dynamics on moduli spaces.
\newblock {\em Bulletin of the American Mathematical Society}, 53(1):41--56,
  2016.

\bibitem{wright2020tour}
Alex Wright.
\newblock A tour through mirzakhani’s work on moduli spaces of riemann
  surfaces.
\newblock {\em Bulletin of the American Mathematical Society}, 57(3):359--408,
  2020.

\bibitem{zorich2006flat}
Anton Zorich.
\newblock Flat surfaces.
\newblock {\em arXiv preprint math/0609392}, 2006.

\end{thebibliography}

\end{document}